\documentclass[11pt,a4paper,twoside,leqno]{amsart}

\usepackage{xifthen}
\usepackage{verbatim}

\newcounter{x}
\newcounter{y}
\newcounter{z}

\newcommand\xaxis{210}
\newcommand\yaxis{-30}
\newcommand\zaxis{90}

\newcommand\topside[3]{
  \fill[fill=cubecolor, draw=black,shift={(\xaxis:#1)},shift={(\yaxis:#2)},
  shift={(\zaxis:#3)}] (0,0) -- (30:1) -- (0,1) --(150:1)--(0,0);
}

\newcommand\leftside[3]{
  \fill[fill=cubecolor, draw=black,shift={(\xaxis:#1)},shift={(\yaxis:#2)},
  shift={(\zaxis:#3)}] (0,0) -- (0,-1) -- (210:1) --(150:1)--(0,0);
}

\newcommand\rightside[3]{
  \fill[fill=cubecolor, draw=black,shift={(\xaxis:#1)},shift={(\yaxis:#2)},
  shift={(\zaxis:#3)}] (0,0) -- (30:1) -- (-30:1) --(0,-1)--(0,0);
}

\newcommand\cube[3]{
  \topside{#1}{#2}{#3} \leftside{#1}{#2}{#3} \rightside{#1}{#2}{#3}
}

\newcommand*\cubecolors[1]{%
  \ifcase#1\relax
  \or\colorlet{cubecolor}{green}%
  \or\colorlet{cubecolor}{green}%
  \or\colorlet{cubecolor}{green}%
  \or\colorlet{cubecolor}{yellow}%
  \or\colorlet{cubecolor}{yellow}%
  \or\colorlet{cubecolor}{yellow}%
  \else
    \colorlet{cubecolor}{yellow}%
  \fi
}

\newcommand\planepartition[1]{
 \setcounter{x}{-1}
  \foreach \a in {#1} {
    \addtocounter{x}{1}
    \setcounter{y}{-1}
    \foreach \b in \a {
      \addtocounter{y}{1}
      \setcounter{z}{-1}
      \foreach \c in {1,...,\b} {
        \addtocounter{z}{1}
        \cubecolors{\c}
        \cube{\value{x}}{\value{y}}{\value{z}}
      }
    }
  }
}

\usepackage[english]{babel}
\usepackage{mathpazo}
\numberwithin{equation}{section}
\usepackage[a4paper,top=3.5cm,bottom=3cm,left=3.5cm,right=3.5cm,bindingoffset=5mm]{geometry}
\usepackage{palatino}  
\usepackage{indentfirst}
\usepackage[colorlinks,bookmarks]{hyperref} 
\usepackage{amsmath,amsthm,amssymb,amsfonts,mathrsfs}
\usepackage{braket}
\usepackage{caption}
\usepackage{tikz}
\usepackage{tikz-cd}
\usetikzlibrary{matrix,shapes,arrows,decorations.pathmorphing}
\tikzset{commutative diagrams/.cd,
mysymbol/.style={start anchor=center,end anchor=center,draw=none}
}
\newcommand\MySymb[2][\square]{%
  \arrow[mysymbol]{#2}[description]{#1}}

\renewenvironment{proof}{{\scshape Proof.}}{\qed}

\definecolor{auburn}{rgb}{0.43, 0.21, 0.1}

\theoremstyle{definition}
\newtheorem{defin}{Definition}[section]
\newtheorem{rem}{Remark}[section]
\newtheorem{conj}{Conjecture}

\newtheoremstyle{conv} 
        {4mm}
        {4mm}
        {\rmfamily}
        {4mm}
        {\itshape}
        {.}
        {1mm}
        {}
\theoremstyle{conv}

\newtheoremstyle{thm} 
        {4mm}
        {4mm}
        {\slshape}
        {4mm}
        {\scshape}
        {.}
        {1mm}
        {}
        
\theoremstyle{thm}
\newtheorem{teo}{Theorem}[section]
\newtheorem*{teo*}{Theorem}
\newtheorem{prop}{Proposition}[section]
\newtheorem{lemma}{Lemma}[section]
\newtheorem{cor}{Corollary}[section]

\newtheoremstyle{rem} 
        {4mm}
        {4mm}
        {\rmfamily}
        {4mm}
        {\scshape}
        {.}
        {2mm}
        {}
\theoremstyle{rem}

\newtheorem{question}{Question}[section]
\newtheorem*{notation*}{Notation}

\newtheorem*{conventions}{Conventions}

\def\be{\begin{equation}}    
\def\ee{\end{equation}}
\def\bitem{\begin{itemize}}
\def\eitem{\end{itemize}}
\def\benum{\begin{enumerate}}
\def\eenum{\end{enumerate}}
\def\vir{\textrm{vir}}
\def\pur{\textrm{pur}}
\def\isom{\cong}  
\def\ra{\rightarrow}
\def\surj{\twoheadrightarrow}        

\DeclareMathOperator{\Quot}{Quot}

\DeclareMathOperator{\Supp}{Supp}
\DeclareMathOperator{\coker}{coker}

\DeclareMathOperator{\DT}{DT}
\DeclareMathOperator{\PT}{PT}
\DeclareMathOperator{\Hilb}{Hilb}
\DeclareMathOperator{\Sym}{Sym}
\DeclareMathOperator{\Ext}{Ext}

\title{Local contributions to Donaldson-Thomas invariants}
\author{Andrea T.~Ricolfi}
\address{Kjell Arholms 41, 4021 Stavanger (Norway)}
\email{andrea.ricolfi@uis.no}


\begin{document}

\begin{abstract}
Let $C$ be a smooth curve embedded in a smooth quasi-projective threefold $Y$, 
and let $Q^n_C=\Quot_n(\mathscr I_C)$ be the Quot scheme of length $n$ quotients of its ideal sheaf.
We show the identity $\tilde\chi(Q^n_C)=(-1)^n\chi(Q^n_C)$, where $\tilde\chi$ is the Behrend weighted Euler characteristic.
When $Y$ is a projective Calabi-Yau threefold, this shows that the DT contribution of a smooth rigid curve
is the signed Euler characteristic of the moduli space. This can be rephrased as a DT/PT wall-crossing type formula, 
which can be formulated for arbitrary smooth curves. 
In general, such wall-crossing formula is shown to be equivalent to a certain Behrend function identity.
\end{abstract}

\maketitle

    \begingroup
    \hypersetup{linkcolor=black}
    \tableofcontents
    \endgroup
    
\hypersetup{
            citecolor=blue,%
            linkcolor=auburn}

\section{Introduction} 
One of the conjectures in~\cite{MNOPpaper} stated that $0$-dimensional Donaldson-Thomas 
(DT, for short) invariants of a smooth projective Calabi-Yau threefold equal the signed Euler characteristic of the moduli space. 
Now, the more general formula
\be\label{zerodimDT}
\tilde \chi(\Hilb^nY)=(-1)^n\chi(\Hilb^nY) 
\ee
is known to hold for \emph{any} smooth threefold $Y$, proper or not~\cite[Thm.~4.11]{BFHilb}. Here $\tilde\chi=\chi(-,\nu)$ is 
the Euler characteristic weighted by the Behrend function~\cite{Beh}. The $0$-dimensional MNOP conjecture is
also solved with cobordism techniques in~\cite{JLI, LEPA}. 

\subsection{Main result}
We propose a statement analogous to \eqref{zerodimDT}, again with no Calabi-Yau or 
properness assumption on the threefold $Y$, but where a curve is present. More precisely, we focus on the space of $1$-dimensional subschemes
$Z\subset Y$ whose fundamental class is the cycle of a fixed Cohen-Macaulay curve $C\subset Y$. A natural scheme structure
on this space seems to be provided by the Quot scheme 
\[
Q^n_C=\Quot_n(\mathscr I_C)
\] 
of $0$-dimensional length $n$ quotients of $\mathscr I_C$, the ideal sheaf of $C$. 
By identifying a surjection $\mathscr I_C\surj F$ with its kernel $\mathscr I_Z$, we see that $Q^n_C$ 
parametrizes curves $Z\subset Y$ differing from $C$ by a finite subscheme of length $n$.
Our main result, proved in \S~\ref{DTinvar}, is the following weighted Euler characteristic computation. 

\begin{teo*}
Let $Y$ be a smooth quasi-projective threefold, $C\subset Y$ a smooth curve. If $Q^n_C=\Quot_n(\mathscr I_C)$, then
\be\label{formula1}
\tilde\chi(Q^n_C)=(-1)^n\chi(Q^n_C).
\ee
\end{teo*}

The proof uses stratification techniques as in~\cite{BFHilb} and~\cite{BB}.

\subsection{Applications}
Let $Y$ be a smooth projective threefold. Let $I_m(Y,\beta)$ be the Hilbert scheme 
of curves $Z\subset Y$ in class $\beta\in H_2(Y,\mathbb Z)$, with 
$\chi(\mathscr O_Z)=m$. Given a Cohen-Macaulay curve $C\subset Y$
of arithmetic genus $g$, embedded in class $\beta$, we show there is a closed immersion
$\iota:Q^n_C\ra I_{1-g+n}(Y,\beta)$. We define
\be\label{inclus}
I_n(Y,C)\subset I_{1-g+n}(Y,\beta)=I
\ee
to be its scheme-theoretic image. When $Y$ is Calabi-Yau, we define
the \emph{contribution} of $C$ to the full (degree $\beta$) DT invariant of $I$ to be the weighted Euler characteristic
\be\label{resnu}
\DT_{n,C}=\chi(I_n(Y,C),\nu_I).
\ee
A first consequence of \eqref{formula1} is the identity
\[
\DT_{n,C}=(-1)^n\chi(I_n(Y,C))
\]
when $C$ is a smooth \emph{rigid} curve in $Y$, because in this case \eqref{inclus} is both open and closed.

\subsubsection{Local DT/PT correspondence}
Let $P_m(Y,\beta)$ be the moduli space of stable pairs introduced by
Pandharipande and Thomas~\cite{PT}. For a Calabi-Yau threefold $Y$ 
and a homology class $\beta\in H_2(Y,\mathbb Z)$, the generating functions 
encoding the DT and PT invariants of $Y$ satisfy the ``wall-crossing type'' formula
\[
\mathsf{DT}_\beta(Y,q)=M(-q)^{\chi(Y)}\cdot \mathsf{PT}_\beta(Y,q).
\]
Here and throughout, $M(q)$ denotes the MacMahon function, the generating series of plane partitions, that is,
\[
M(q)=\sum_{\pi}q^{|\pi|}=\prod_{k\geq 1}(1-q^k)^{-k}.
\]
The DT/PT correspondence stated above was first conjectured in~\cite{PT}~and later proved in~\cite{Bri,Toda1}.
In this paper we ask about a similar formula relating the \emph{local} invariants, that is, the contributions
of a single smooth curve $C\subset Y$ to the full DT and PT invariants of $Y$ in the class $\beta=[C]$.

\medskip
If $C\subset Y$ is a fixed smooth curve of genus $g$, we consider the closed subscheme
\[
P_n(Y,C)\subset P_{1-g+n}(Y,\beta)=P
\]
of stable pairs with Cohen-Macaulay support equal to $C$. 
We use~\eqref{formula1} and the isomorphism $P_n(Y,C)\isom \Sym^nC$ to show the generating function identity
\be\label{genfunctions}
\sum_{n\geq 0}\tilde\chi(I_n(Y,C))q^n=M(-q)^{\chi(Y)}(1+q)^{2g-2},
\ee
which holds without any Calabi-Yau assumption.

For $Y$ a Calabi-Yau threefold, we consider the stable pair local contributions
\[
\PT_{n,C}=\chi(P_n(Y,C),\nu_P)
\]
like we did in \eqref{resnu} for ideal sheaves. We assemble all the local invariants into generating functions 
\begin{align*}
\mathsf{DT}_C(q)&=\sum_{n\geq 0}\DT_{n,C}q^n\\
\mathsf{PT}_C(q)&=\sum_{n\geq 0}\PT_{n,C}q^n.
\end{align*}
The PT side has been computed~\cite[Lemma~3.4]{BPS} and the result is 
\[
\mathsf{PT}_C(q)=n_{g,C}\cdot (1+q)^{2g-2},
\]
where $n_{g,C}$ is the BPS number of $C$.
Therefore it is clear by looking at \eqref{genfunctions} that the DT/PT correspondence
\be\label{zmxncb}
\mathsf{DT}_C(q)=M(-q)^{\chi(Y)}\cdot \mathsf{PT}_C(q)
\ee
holds for $C$ if and only if, for every $n$, one has
\[
\DT_{n,C}=n_{g,C}\cdot \tilde\chi(I_n(Y,C)).
\]
For instance, it holds when $C$ is rigid. In the last section, we 
discuss the plausibility to conjecture the identity \eqref{zmxncb} 
to hold for all smooth curves.

\begin{conventions}
All schemes are defined over $\mathbb C$, and all threefolds are assumed to be smooth. 
An \emph{ideal sheaf} is a torsion-free sheaf with rank one and trivial
determinant. For a smooth projective threefold $Y$, we denote by $I_m(Y,\beta)$ the moduli space of ideal sheaves with 
Chern character $(1,0,-\beta,-m+\beta\cdot c_1(Y)/2)$. It is naturally isomorphic to the Hilbert scheme parametrizing
closed subschemes $Z\subset Y$ of codimension at least $2$, with homology class $\beta$ and $\chi(\mathscr O_Z)=m$. 
A \emph{Cohen-Macaulay} curve is a scheme of \emph{pure} dimension one without embedded points. 
The \emph{Calabi-Yau} condition for us is simply the existence of a trivialization $\omega_Y\isom \mathscr O_Y$.
We use the word \emph{rigid} as a shorthand for the more correct \emph{infinitesimally rigid}: for a smooth embedded 
curve $C\subset Y$, this means $H^0(C,N_{C/Y})=0$, where $N_{C/Y}$ is the normal bundle. 
Finally, we refer to~\cite{Beh} for the definition and properties of the Behrend function and of the weighted Euler characteristic.
\end{conventions}

\section{The local model}\label{sec:localmodel}

The global geometry of a fixed smooth curve in a threefold $C\subset Y$ 
will be analysed through the local model
\[
\mathbb A^1\subset \mathbb A^3
\]
of a line in affine space. 
We get started by introducing the moduli space of ideal sheaves for this local model.

Let $X$ be the resolved conifold, i.e.~the total space of the rank two bundle
$\mathscr O_{\mathbb P^1}(-1,-1)\ra\mathbb P^1$. It is a quasi-projective Calabi-Yau threefold.
We let $C_0\subset X$ be the zero section, and $\mathbb A^3\subset X$ a \emph{fixed} chart of the bundle.

\begin{defin}\label{def:locmodel}
For any integer $n\geq 0$, we define
\[
M_n\subset I_{n+1}(X,[C_0])
\]
to be the open subscheme parametrizing ideal sheaves $\mathscr I_Z\subset \mathscr O_X$ such 
that no associated point of $Z$ is contained in $X\setminus \mathbb A^3$. 
\end{defin}

Since $C_0$ is rigid, we can interpret $M_n$ as the moduli space of ``curves'' in $\mathbb A^3$,
consisting of a \emph{fixed} affine line $L=C_0\cap\mathbb A^3$ together with $n$ roaming points.

The scheme $M_n$ seems to be the perfect local playground for 
studying the enumerative geometry of a fixed curve (with $n$ points) in a threefold.
Exactly like $\Hilb^n\mathbb A^3$ was essential~\cite{BFHilb} to unveil the Donaldson-Thomas 
theory of $\Hilb^nY$, where $Y$ is any Calabi-Yau threefold, the space $M_n$
will help us to figure out the DT contribution of a fixed smooth rigid curve in a Calabi-Yau threefold (and, conjecturally, all smooth curves).
Forgetting about the Calabi-Yau assumption, we will find out that understanding the local picture in $\mathbb A^3$
gives information about \emph{arbitrary} threefolds, in perfect analogy with the results of~\cite{BFHilb}.

\medskip
In the rest of this section, we show that $M_n$ is isomorphic to the Quot scheme of 
the ideal sheaf of a line, and we compute its DT invariant via equivariant localization.

Let $L$ denote the line $C_0\cap \mathbb A^3$. Note that if $Z\subset X$ corresponds to a point of $M_n$, by definition its embedded points can only be supported on $L$. Similarly, isolated points are confined to the chart $\mathbb A^3\subset X$.

\begin{prop}\label{identhgf}
There is an isomorphism of schemes $M_n\isom \Quot_n(\mathscr I_L)$.
\end{prop}

\begin{proof}
Let $T$ be a scheme and let $\iota:\mathbb A^3\times T\ra X\times T$ be the natural open immersion. 
If $\mathscr O_{X\times T}\surj\mathscr O_{\mathcal Z}$ represents a $T$-valued point of $M_n$, we can consider the sheaf
$\mathscr F=\mathscr I_{C_0\times T}/\mathscr I_{\mathcal Z}$, which by definition of $M_n$ is supported on a subscheme 
of $\mathbb A^3\times T$ which is finite of relative length $n$ over $T$. Restricting the short exact sequence
\[
0\ra \mathscr F\ra \mathscr O_{\mathcal Z}\ra \mathscr O_{C_0\times T}\ra 0
\]
to $\mathbb A^3\times T$ gives a short exact sequence
\[
0\ra \iota^\ast\mathscr F\ra \iota^\ast\mathscr O_{\mathcal Z}\ra \mathscr O_{L\times T}\ra 0
\]
with $T$-flat kernel, so we get a $T$-valued point $\mathscr I_{L\times T}\surj \iota^\ast\mathscr F$ of $\Quot_n(\mathscr I_L)$, 
since as we already noticed $\iota^\ast\mathscr F$ has the same support as $\mathscr F$.

Conversely, a $T$-flat quotient $\mathscr F$ of the ideal sheaf $\mathscr I_{L\times T}$ determines a flat family of subschemes
\[
\mathcal Z\subset \mathbb A^3\times T\ra T,
\]
where $L\times T\subset \mathcal Z$. Taking closures inside $X\times T$, we get closed immersions 
\[
C_0\times T\subset \overline{\mathcal Z}\subset X\times T. 
\]
The support of $\mathscr F$ is proper
over $T$, and since $\mathbb A^3$ and $X$ are separated, we see that the inclusion maps of $\Supp\mathscr F$ 
in $\mathbb A^3\times T$ and $X\times T$ are proper. This says that the pushforward $\iota_\ast\mathscr F$ is a coherent sheaf
on $X\times T$. It agrees with the relative ideal of the immersion $C_0\times T\subset \overline{\mathcal Z}$, and is supported
exactly where $\mathscr F$ is. Finally, the short exact sequence
\[
0\ra \iota_\ast\mathscr F\ra \mathscr O_{\overline{\mathcal Z}}\ra \mathscr O_{C_0\times T}\ra 0
\]
says $\mathscr O_{\overline{\mathcal Z}}$ is $T$-flat (being an extension of $T$-flat sheaves), therefore we get a $T$-valued point of $M_n$.
The two constructions are inverse to each other, whence the claim.
\end{proof}

\medskip
Keeping the above result in mind, we will sometimes silently identify $M_n$ with $\Quot_n(\mathscr I_L)$, and 
we will switch from subschemes (or ideal sheaves) to quotient sheaves with no further mention.

\begin{rem}
The resolved conifold $X$ plays little role here. In fact, the above proof shows the following. If there is an immersion
$\mathbb A^3\ra Y$ into some Calabi-Yau threefold $Y$, such that the closure of a line $L\subset \mathbb A^3$ 
becomes a rigid rational curve $C\subset Y$, then the Hilbert scheme $I_{n+1}(Y,[C])$ contains an open 
subscheme isomorphic to $\Quot_n(\mathscr I_L)$.
\end{rem}

\subsection{The DT invariant}\label{sectiontwo}
The open subscheme $M_n\subset I_{n+1}(X,[C_0])$ inherits, by restriction, 
a torus-equivariant symmetric obstruction theory, and therefore an equivariant virtual fundamental class
\[
\bigl[M_n\bigr]^\vir\in A_0^{\mathbf T}(M_n)\otimes \mathbb Q(s_1,s_2,s_3).
\]
The torus $\mathbf T\subset (\mathbb C^{\times})^3$ we are referring to is the two-dimensional torus 
fixing the Calabi-Yau form on $X$, and acting on $X$ by rescaling coordinates. 
We refer the reader to~\cite[\S~2.3]{BB} for more details on this action and for an accurate description of the fixed locus 
\[
I_{m}(X,d[C_0])^{\mathbf T}\subset I_{m}(X,d[C_0])
\]
for every $d>0$. An ideal sheaf $\mathscr I_Z\in M_n$ is $\mathbf T$-fixed
if it becomes a monomial ideal when restricted to the chosen chart $\mathbb A^3\subset X$. 
The fixed locus $M_n^{\mathbf T}\subset M_n$ is isolated and reduced, by~\cite[Lemma $6$ and $8$]{MNOPpaper}. 
In the language of the topological vertex, a $\mathbf T$-fixed ideal can be described as a 
way of stacking $n$ boxes in the corner of the one-legged configuration $(\emptyset, \emptyset, \square)$.
We give an example in Figure~\ref{fig:M1}.

\begin{figure}[h]
\captionsetup{width=0.68\textwidth}
\centering
\begin{tikzpicture}[scale=0.26]
\planepartition{{9,3,3,2},{3,2,1,1},{2,2},{1}}
\end{tikzpicture}
\caption{\small{A $\mathbf T$-fixed ideal in $M_n$.
The ``$z$-axis'' has to be figured as infinitely long,
corresponding to the line $L=C_0\cap \mathbb A^3$.}} \label{fig:M1}
\end{figure}
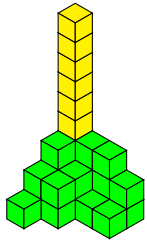

The parity of the tangent space dimension at torus-fixed points of $I_{m}(X,d[C_0])$ was computed in~\cite[Prop.~2.7]{BB}. 
The result is $(-1)^{m-d}$ by an application of~\cite[Thm.~2]{MNOPpaper}. 
In our case $m=n+1$ and $d=1$ so we get the sign $(-1)^n$ for $I_{n+1}(X,[C_0])$.
Since $M_n$ is open in this Hilbert scheme, the parity does not change and we deduce that
\[
(-1)^{\dim T_{M_n}|_{\mathscr I}}=(-1)^n
\]
for all fixed points $\mathscr I\in M_n^{\mathbf T}$. 
After the Calabi-Yau specialization $s_1+s_2+s_3=0$ of the equivariant parameters, 
and by the symmetry of the obstruction theory, the virtual localization formula~\cite{GP} reads
\be\label{virloc}
\bigl[M_n\bigr]^\vir=(-1)^n\bigl[M_n^{\mathbf T}\bigr]\in A_0(M_n),
\ee
where, as mentioned above, the sign 
\[
(-1)^n=\frac{e^{\mathbf T}(\Ext^2(\mathscr I,\mathscr I))}{e^{\mathbf T}(\Ext^1(\mathscr I,\mathscr I))}\in \mathbb Q(s_1,s_2,s_3)
\]
comes from~\cite[Thm.~2]{MNOPpaper}.

We define the Donaldson-Thomas invariant of $M_n$ by equivariant localization through formula \eqref{virloc}. 
Hence we can compute it as
\[
\DT(M_n)=(-1)^n\chi(M_n),
\]
where the Euler characteristic $\chi(M_n)$ counts the number of fixed points.

It is easy to see (see for instance the proof of~\cite[Lemma~2.9]{BB}) that
\be\label{localeuler}
\sum_{n\geq 0}\chi(M_n)q^n=\frac{M(q)}{1-q}
\ee
where $M(q)=\prod_{m\geq 1}(1-q^m)^{-m}$ is the MacMahon function, the generating series of plane partitions.
In particular, the DT partition function for the moduli spaces $M_n$ takes the form
\[
\sum_{n\geq 0}\DT(M_n)q^{n+1}=q\frac{M(-q)}{1+q}=q(1-2q+5q^2-11q^3+\cdots).
\]
In the sum, we have switched indices by one to follow the general convention of weighting the variable $q$ by the holomorphic Euler characteristic.

\section{Curves and Quot schemes}\label{sec:quot}

\subsection{Main characters}
Let $C$ be a Cohen-Macaulay curve embedded in a quasi-projective variety $Y$ and let $\mathscr I_C\subset\mathscr O_Y$ denote its ideal sheaf. 
For an integer $n\geq 0$, 
let $Q=\Quot_n(\mathscr I_C)$ be the Quot scheme parametrizing $0$-dimensional quotients of $\mathscr I_C$, of length $n$.
See~\cite{Nit} for a proof of the representability of the Quot functor in the quasi-projective case. 
By looking at the full exact sequence
\[
0\ra \mathscr I_Z\ra \mathscr I_C\ra F\ra 0
\] 
for a given point $[\mathscr I_C\surj F]$ of $Q$, we think of the Quot scheme as parametrizing curves 
$Z\subset Y$ obtained from $C$, roughly speaking, by adding a finite subscheme of length $n$. 

\begin{defin}\label{def:supponcurve}
We denote by $W^n_C\subset Q$
the closed subset parametrizing quotients $\mathscr I_C\surj F$ such that $\Supp F\subset C$, where
$\Supp F$ denotes the \emph{set-theoretic} support of the sheaf $F$. 
We endow $W^n_C$ with the reduced scheme structure.
\end{defin}

Given a point $[F]\in W^n_C$, the support of $F$ has the structure of a closed subscheme of $Y$ but not 
of $C$ in general; however, $\Supp F$ defines naturally an effective zero-cycle on $C$. 
Sending $[F]$ to this cycle is a morphism, as we now show.

\begin{lemma}\label{quottosym}
There is a natural morphism $u:W^n_C\ra \Sym^nC$ sending a quotient to the corresponding zero-cycle. 
\end{lemma}

\begin{proof}
Let $T$ be a reduced scheme, which we take as the base of a valued point $\mathscr I_{C\times T}\surj \mathscr F$ of $W^n_C$.
Let $\pi:Y\times T\ra T$ be the projection. Working locally on $Y$ and $T$ we see that by Nakayama's lemma, 
$\Supp \mathscr F\cap \pi^{-1}(t)=\Supp \mathscr F_t$ for every closed point $t\in T$.
Then the closed subscheme $\Supp \mathscr F\subset Y\times T$ is flat over $T$
(because the Hilbert polynomial of the fibres $\Supp \mathscr F_t$ is the constant $n$ and $T$ is reduced),
and hence defines a valued point $T\ra \Hilb^nY$. Composing with the Hilbert-Chow map $\Hilb^nY\ra \Sym^nY$ we get a morphism
$T\ra \Sym^nY$ which factors through $\Sym^nC$, by definition of $W^n_C$.
\end{proof}

\medskip
For every partition $\alpha$ of $n$ there is a locally closed subscheme
\[
\Sym^n_\alpha C\subset \Sym^nC
\]
parametrizing zero-cycles with multiplicities dictated by $\alpha$. 
These subschemes form a stratification of $\Sym^nC$, which we can use together with the morphism $u$
to stratify $W^n_C$ by locally closed subschemes
\be\label{def:strat}
W^\alpha_C=u^{-1}(\Sym^n_\alpha C)\subset W^n_C.
\ee
In particular, since $\Sym^n_{(n)}C\isom C$, there is a natural morphism 
\be\label{map:zlt}
\pi_C:W^{(n)}_C\ra C
\ee
corresponding to the deepest stratum.

The main result of this section asserts that, when $C$ is a smooth curve and $Y$ is a smooth threefold,
the map \eqref{map:zlt} is a Zariski locally trivial fibration. 
The proof is based on the Quot scheme adaptation of the 
results proven by Behrend and Fantechi for $\Hilb^nY$~\cite[\S~4]{BFHilb}.

Let us now introduce what will turn out to be the typical fibre of $\pi_C$. Recall that
$X$ denotes the resolved conifold and $C_0\subset X$ is the zero section.

\begin{defin}\label{def:punct}
We denote by $F_n\subset M_n$ the closed subset parametrizing subschemes $Z\subset X$ such that the 
relative ideal $\mathscr I_{C_0}/\mathscr I_Z$ is entirely supported at the origin $0\in L=C_0\cap\mathbb A^3$.
We use the shorthand
\[
\nu_n=\nu_{M_n}\big{|}_{F_n}
\]
for the restriction of the Behrend function on $M_n$ to $F_n$. 
\end{defin}

We can think of $F_n$ and all strata $W^\alpha_C\subset W^n_C$ as endowed with the reduced scheme structure.

\begin{rem}
The morphism $u:W^n_C\ra \Sym^nC$ plays the role of the Hilbert-Chow map $\Hilb^nY\ra \Sym^nY$ in the $0$-dimensional setting, and the subscheme $F_n\subset M_n$ is the analogue of the punctual Hilbert scheme 
$\Hilb^n(\mathbb A^3)_0\subset \Hilb^n \mathbb A^3$
parametrizing finite subschemes supported at the origin.
\end{rem}

\begin{prop}\label{yuhjbn}
There is a natural isomorphism $W^{(n)}_L=L\times F_n$. Moreover, if $p:W^{(n)}_L\ra F_n$ is the projection, we have the relation
\be\label{nufunctions}
\nu_{M_n}\big{|}_{W^{(n)}_L}=p^\ast \nu_n.
\ee
\end{prop}

\begin{proof}
We view $L$ as the additive group $\mathbb G_a$ and we let it act on itself by translation. 
This induces an action of $L$ on $M_n$. Restricting this action to $F_n$ gives a map 
\[
L\times F_n\ra W^{(n)}_L.
\] 
This is an isomorphism, whose inverse is the morphism $\pi_L\times \rho:W^{(n)}_L\ra L\times F_n$, where 
\[
\rho:W^{(n)}_L\ra F_n
\]
takes a subscheme $[Z]\in W^{(n)}_L$ to its translation by $-x\in \mathbb G_a$, where $x\in L=\mathbb G_a$ is the
unique embedded point on $Z$.
The identity \eqref{nufunctions} follows because the Behrend function is constant 
on orbits and for each $P\in F_n$ the slice $L\times \{P\}$ is isomorphic to an orbit.
\end{proof}

\subsection{Comparing Quot schemes}

Let $\varphi:Y\ra Y'$ be a morphism of varieties, where $Y$ is quasi-projective and $Y'$ is complete.
Let $C'\subset Y'$ be a Cohen-Macaulay curve and let $C=\varphi^{-1}(C')\subset Y$ denote its preimage.
We assume $C$ is a Cohen-Macaulay curve and $C'$ is its scheme-theoretic image.
In Lemma \ref{rem:sti} we give sufficient conditions for this to hold.

Given an integer $n\geq 0$, we let $Q=\Quot_n(\mathscr I_C)$ and $Q'=\Quot_n(\mathscr I_{C'})$.

\medskip
We will show how to associate to these data a rational map
\[
\Phi:Q\dashrightarrow Q'.
\]
The rough idea is that we would like to ``push down'' the $n$ points in the support of a sheaf $[F]\in Q$
and still get $n$ points, which would ideally form the support of the image sheaf $\varphi_\ast F$. 
This only works, as one might expect, over the open subscheme $V\subset Q$ parametrizing sheaves 
$F$ such that $\varphi|_{\Supp F}$ is injective. Moreover, the resulting map 
$\Phi:V\ra Q'$ turns out to be \'etale whenever $\varphi$ is. After extending this result to quasi-projective $Y'$, 
we will be able to compare $\Quot_n(\mathscr I_C)$ with the local picture of $M_n=\Quot_n(\mathscr I_L)$, 
and pull back (\'etale-locally) the known results about $\pi_L$ (Proposition~\ref{yuhjbn}) to deduce 
that the maps $\pi_C$ defined in~\eqref{map:zlt} are Zariski locally trivial, at least when $C$ and $Y$ are smooth.

\begin{lemma}\label{rem:sti}
Let $\varphi:Y\ra Y'$ be an \'etale morphism of varieties with image $U$. 
If $C'\subset Y'$ is a Cohen-Macaulay curve and $U\cap C'$ is dense in $C'$, 
then $C=\varphi^{-1}(C')$ is Cohen-Macaulay and $C'$ is its scheme-theoretic image.
\end{lemma}

Before proving the lemma, recall that a closed subscheme $C'$ of a scheme $Y'$ is said to have an embedded component if there is a dense open subset 
$U\subset Y'$ such that $U\cap C'$ is dense in $C'$ but its scheme-theoretic closure does not equal $C'$ scheme-theoretically. 
Recall that a curve is Cohen-Macaulay if it has no embedded points.

\medskip
\begin{proof}
Since the restriction $C\ra C'$ is \'etale and $C'$ is Cohen-Macaulay, $C$ is also Cohen-Macaulay.
Moreover, $U$ is open (because $\varphi$ is \'etale) and dense (because $Y'$ is irreducible), and since $U\cap C'\subset C'$ is dense, 
the scheme-theoretic closure of $U\cap C'$ agrees with $C'$ topologically.
But since $C'$ has no embedded points, they in fact agree as schemes. 
On the other hand, the open subset $U\cap C'\subset C'$ is the set-theoretic image of the \'etale map $C\ra C'$.
Therefore its scheme-theoretic closure is the scheme-theoretic image of $C\ra C'$. 
So $C'$ is the scheme-theoretic image of $C$.
\end{proof}

\begin{notation*}
For a scheme $S$, we will denote $\varphi_S=\varphi\times \textrm{id}_S:Y\times S\ra Y'\times S$.
The case $S=Q$ being quite special, we will let $\tilde\varphi$ denote $\varphi_Q=\varphi\times \textrm{id}_Q$. 
\end{notation*}

By our assumptions, $C'\times S$ is the scheme-theoretic image of $C\times S\subset Y\times S$ under $\varphi_S$, for any scheme $S$.
Indeed, $\varphi$ is quasi-compact so the scheme-theoretic image commutes with flat base change.

\begin{rem}\label{alskdj}
Let $\mathscr E$ be the universal sheaf on $Q$, with scheme-theoretic support $\Sigma\subset Y\times Q$. Since $\Sigma\ra Q$ is proper (by the very 
definition of the Quot functor), and it factors through the (separated) projection $\pi:Y'\times Q\ra Q$, necessarily the map 
$\Sigma\ra Y'\times Q$ must be proper. Since $\tilde\varphi_\ast\mathscr E$ is obtained as a pushforward from $\Sigma$, it is coherent.
Therefore, pushing forward coherent sheaves supported on $\Sigma$ will still give us coherent sheaves, even if $\varphi$ is not proper.
\end{rem}

\begin{rem}
Let $[F]\in Q$ be any point, and let $\mathscr I_Z\subset \mathscr I_C$ be the kernel of the surjection.
Then we have closed immersions $C\subset Z\subset Y$ and $C'\subset Z'\subset Y'$, where $Z'$ denotes the scheme-theoretic image of $Z$.
Using that $R^1\varphi_\ast F=0$, we find a commutative diagram of coherent $\mathscr O_{Y'}$-modules
\[
\begin{tikzcd}
0\arrow{r} & 
\mathscr I_{C'}/\mathscr I_{Z'}\arrow{r}\arrow[hook]{d} & 
\mathscr O_{Z'}\arrow{r}\arrow[hook]{d} &
\mathscr O_{C'}\arrow{r}\arrow[hook]{d} &
0\\
0\arrow{r} &
\varphi_\ast F\arrow{r} &
\varphi_\ast\mathscr O_Z\arrow{r} &
\varphi_\ast\mathscr O_C\arrow{r} &
0
\end{tikzcd}
\]
having exact rows. The middle and right vertical arrows are monomorphisms by definition of scheme-theoretic image. For instance,
\[
\mathscr I_{C'}=\ker\bigl(\mathscr O_{Y'}\surj \mathscr O_{C'}\bigr)=\ker\bigl(\mathscr O_{Y'}\surj \mathscr O_{C'}\ra \varphi_\ast\mathscr O_C\bigr)
\]
implies that $\mathscr O_{C'}\ra \varphi_\ast\mathscr O_C$ is injective.
\end{rem}

The previous remark can be made universal. Let $\mathscr I_{C\times Q}\surj \mathscr E$
be the universal quotient, living over $Y\times Q$. Looking at its kernel $\mathscr I_{\mathcal Z}$, 
we get a commutative diagram
\[
\begin{tikzcd}
C\times Q\arrow{d}\arrow[hook]{r} & 
\mathcal Z\arrow{d}\arrow[hook]{r} &
Y\times Q\arrow{d}{\tilde\varphi}\\
C'\times Q\arrow[hook]{r} &
\mathcal Z'\arrow[hook]{r} &
Y'\times Q
\end{tikzcd}
\]
where the horizontal arrows are closed immersions, $\tilde\varphi=\varphi\times \textrm{id}_Q$ and
$\mathcal Z'$ denotes the scheme-theoretic image of $\mathcal Z$.
We also get a commutative diagram of coherent $\mathscr O_{Y'\times Q}$-modules
\[
\begin{tikzcd}
0\arrow{r} & 
\mathscr I_{C'\times Q}/\mathscr I_{\mathcal Z'}\arrow{r}\arrow[hook]{d} & 
\mathscr O_{\mathcal Z'}\arrow{r}\arrow[hook]{d} &
\mathscr O_{C'\times Q}\arrow{r}\arrow[hook]{d} &
0\\
0\arrow{r} &
\tilde\varphi_\ast\mathscr E\arrow{r} &
\tilde\varphi_\ast\mathscr O_{\mathcal Z}\arrow{r} &
\tilde\varphi_\ast\mathscr O_{C\times Q}\arrow{r} &
0
\end{tikzcd}
\]
having exact rows.

\medskip

Let us consider the composition
\be\label{grhtjy}
\alpha:\mathscr I_{C'\times Q}\twoheadrightarrow \mathscr I_{C'\times Q}/\mathscr I_{\mathcal Z'}\hookrightarrow \tilde\varphi_\ast\mathscr E
\ee
and let us write $\mathscr K$ for its cokernel. 
By Remark \ref{alskdj}, $\tilde\varphi_\ast\mathscr E$ is coherent, hence $\mathscr K=\coker\alpha$ is coherent, too. 
Thus $\Supp \mathscr K$ is closed in $Y'\times Q$. Since $Y'$ is complete, the projection $\pi:Y'\times Q\ra Q$ is closed. Therefore the complement
\be\label{definofu}
Q\setminus \pi(\Supp \mathscr K)\subset Q
\ee
is an open subset of $Q$.

\begin{prop}\label{prop:big}
Let $[F]\in Q$ be a point such that $\varphi$ is \'etale in a neighborhood of $\Supp F$ and 
$\varphi(x)\neq\varphi(y)$ for all distinct points $x,y\in\Supp F$.
Then there is an open neighborhood $U\subset Q$ of $[F]$ admitting an \'etale map $\Phi:U\ra Q'$.
\end{prop}

\begin{proof}
We first observe that we may reduce to prove the result after restricting $Y$ to \emph{any} open neighborhood of $\Supp F$ inside $Y$.
Indeed, if $V$ is any such neighborhood, $\Quot_n(\mathscr I_C|_V)$ is an open subscheme of $Q$ that still contains $[F]$ as a point. 
We will take advantage of this freedom by choosing a suitable $V$. We divide the proof in two steps.

\medskip
\emph{Step $1$}:\emph{ Existence of the map.}
Let $Z\subset Y$ be the closed subscheme determined by the kernel of $\mathscr I_C\surj F$.
Let $Z'\subset Y'$ be its scheme-theoretic image. 
Since $\varphi|_{\Supp F}$ is injective and $\varphi$ is \'etale around $\Supp F$, the natural monomorphism 
$\mathscr I_{C'}/\mathscr I_{Z'}\ra \varphi_\ast F$ is an isomorphism and $\varphi_\ast F$ is a sheaf of length $n$, 
so that we get a well-defined point
\be\label{pointtt}
[\varphi_\ast F]\in Q'.
\ee
Now let $B\subset Y$ denote the support of $F$ and let $V$ be an open 
neighborhood of $B$ such that $\varphi$ is \'etale when restricted to $V$. We may assume $V$ is affine, and in fact we may also assume
$Y=V$, by our initial remark. 

In this situation, we have the Cartesian square
\[
\begin{tikzcd}[row sep=large, column sep=large]
Y\times [F]\MySymb{dr}\arrow[swap]{d}{\varphi} \arrow[hook]{r}{i} & Y\times Q \arrow{d}{\tilde\varphi}\\
Y'\times [F] \arrow[hook]{r}{j} & Y'\times Q
\end{tikzcd}
\]
where the map $\tilde\varphi$ is affine. Therefore, working affine-locally on $Y'\times Q$, we see that the natural base change map
$j^\ast\tilde\varphi_\ast\mathscr E\,\,\widetilde{\ra}\,\,\varphi_\ast F$ is an isomorphism. This proves that the surjection
$\mathscr I_{C'}\surj \varphi_\ast F$ defining the point \eqref{pointtt} is obtained precisely restricting
$\alpha:\mathscr I_{C'\times Q}\ra\tilde\varphi_\ast\mathscr E$, defined in \eqref{grhtjy}, to the slice 
\[
j:Y'\times [F]\subset Y'\times Q.
\]
Letting $U\subset Q$ denote the open subset defined in \eqref{definofu}, we see that $\alpha$ restricts to a surjection
\[
\alpha|_{Y'\times U}:\mathscr I_{C'\times U}\surj \varphi_{U\ast}\mathscr E_U, 
\]
where $\mathscr E_U=\mathscr E|_{Y\times U}$. The target is a coherent sheaf, and it is flat over $U$.
Indeed, $\mathscr E$ is flat over $Q$, thus $\tilde\varphi_\ast \mathscr E$ is also flat over $Q$. But  
$\varphi_{U\ast}\mathscr E_U$ is naturally isomorphic to the pullback of $\tilde\varphi_\ast \mathscr E$ along the open immersion
$Y'\times U\subset Y'\times Q$, therefore it is flat over $U$. Finally, the map $\alpha|_{Y'\times U}$ restricts to length $n$ quotients
\[
\mathscr I_{C'}\surj \varphi_\ast E,
\]
for any closed point $[E]\in U$. Therefore we have just constructed a morphism
\[
\Phi:U\ra Q',\qquad [E]\mapsto [\varphi_\ast E].
\]

\emph{Step $2$}:\emph{ Proving it is \'etale.} 
We may shrink $Y$ further and replace it by any affine open neighborhood of $B=\Supp F$ contained in $Y\setminus A$, 
where $A$ is the closed subset
\[
A=\coprod_{b\in B}\varphi^{-1}\varphi(b)\setminus\{b\}\subset Y.
\]
After this choice, the preimage $Y_{\varphi(b)}$ is the single point $\{b\}$, for every $b\in B$. 
This condition implies that the natural morphism
\be\label{pushpull}
\varphi^\ast\varphi_\ast F\,\,\widetilde{\ra}\,\,F
\ee
is an isomorphism. Although this condition is not preserved in any open neighborhood of $[F]$, it is preserved infinitesimally, which
is exactly what we need to establish \'etaleness.

We now use the infinitesimal criterion to show $\Phi$ is \'etale at the point $[F]$.
Let $\iota:T\ra \overline T$ be a small extension of fat points. 
Assume we have a commutative square
\[
\begin{tikzcd}[row sep=large, column sep=large]
T\arrow[hook]{r}{\iota}\arrow[swap]{d}{g} & \overline T\arrow{d}{h}\arrow[dotted]{dl}[description]{v}\\
U\arrow{r}{\Phi} & Q'
\end{tikzcd}
\]
where $g$ sends the closed point $0\in T$ to $[F]$. Then we want to find a \emph{unique} arrow 
$v$ making the two induced triangles commutative. Rephrasing this in terms of families of sheaves, let
$\mathscr I_{C\times T}\surj \mathscr G$ and $\mathscr I_{C'\times \overline T}\surj \mathscr H$
be the families corresponding to $g$ and $h$, living over $Y\times T$ and $Y'\times \overline T$ respectively. We are after a unique $U$-valued family $\mathscr I_{C\times \overline T}\surj \mathscr V$
over $Y\times \overline T$ with the following properties.
\bitem
\item [$(\star)$] The condition $\Phi\circ v=h$ means we can find a commutative diagram
\[
\begin{tikzcd}
\mathscr I_{C'\times\overline T}\arrow[two heads]{r}\arrow[swap]{d}{\textrm{id}} &\varphi_{\overline T\ast}\mathscr V\arrow{d}{\simeq}\\
\mathscr I_{C'\times \overline T}\arrow[two heads]{r} & \mathscr H
\end{tikzcd}\qquad \textrm{of sheaves on }Y'\times\overline T.
\]

Let us explain the condition in detail. We use, in the following, the notation $\tilde p=1_Y\times p$ and $\overline p=1_{Y'}\times p$, for a given map $p$.
Looking at the diagram
\[
\begin{tikzcd}[row sep=large, column sep=large]
Y\times \overline T\arrow[swap]{d}{\tilde v}\arrow{r}{\varphi_{\overline T}} &
Y'\times \overline T\arrow{d}{\overline v} \\
Y\times U\arrow[swap,hook]{d}{} \arrow{r}{\varphi_U} &
Y'\times U\arrow{d}{\overline \Phi} \\
Y\times Q & Y'\times Q'
\end{tikzcd}
\]
we should require
\[
\mathscr H\isom \overline v^\ast \overline\Phi^\ast \mathscr E',
\]
where $\mathscr E'$ is the universal quotient sheaf on $Y'\times Q'$. However, 
\[
\overline v^\ast \overline\Phi^\ast \mathscr E'\isom \overline v^\ast\varphi_{U\ast}\mathscr E_U\isom \varphi_{\overline T\ast}\mathscr V,
\]
where we have used ``affine base change'' again. 
\item [$(\star\star)$] Looking at
\[
\begin{tikzcd}[row sep=large, column sep=large]
Y\times T\MySymb{dr}\arrow[hook,swap]{d}{\tilde \iota} \arrow{r}{\varphi_T} & Y'\times T\arrow[hook]{d}{\overline \iota}\\
Y\times \overline T\arrow{r}{\varphi_{\overline T}} & Y'\times \overline T,
\end{tikzcd}
\]
the condition $v\circ\iota=g$ means we can find a commutative diagram 
\[
\begin{tikzcd}
\tilde \iota^\ast \mathscr I_{C\times\overline T}\arrow[two heads]{r}\arrow[swap]{d}{\textrm{id}} & \tilde \iota^\ast \mathscr V\arrow{d}{\simeq}\\
\mathscr I_{C\times T}\arrow[two heads]{r} & \mathscr G
\end{tikzcd}\qquad \textrm{of sheaves on }Y\times T.
\]
\eitem

We observe that 
\bitem
\item [(i)] the isomorphism $\varphi_{\overline T\ast}\mathscr V\,\,\widetilde{\ra}\,\, \mathscr H$ defining $(\star)$, and
\item [(ii)] the isomorphism $\varphi_{\overline T}^\ast\varphi_{\overline T\ast}\mathscr V\,\,\widetilde{\ra}\,\,\mathscr V$, 
the ``infinitesimal thickening'' of \eqref{pushpull},
\eitem
together determine $v$ uniquely: it is the \emph{unique} arrow corresponding to the isomorphism class of the surjection
\[
\mathscr I_{C\times\overline T}=\varphi_{\overline T}^\ast \mathscr I_{C'\times\overline T}\surj
\varphi_{\overline T}^\ast\mathscr H=\mathscr V. 
\]
To check that condition $(\star\star)$ is fulfilled by this family,
we use that $\Phi\circ g=h\circ \iota$. In other words, there is a commutative diagram
\[
\begin{tikzcd}
\overline\iota^\ast \mathscr I_{C'\times\overline T}\arrow[two heads]{r}\arrow[swap]{d}{\textrm{id}} & \overline\iota^\ast\mathscr H\arrow{d}{\simeq}\\
\mathscr I_{C'\times T}\arrow[two heads]{r} & \varphi_{T\ast}\mathscr G
\end{tikzcd}\qquad \textrm{of sheaves on }Y'\times T.
\]
As before, we have noted that the family corresponding to $\Phi\circ g$ is 
\[
\overline g^\ast\varphi_{U\ast}\mathscr E_U\isom \varphi_{T\ast}\mathscr G,
\]
where $\overline g$ is the map $\textrm{id}_{Y'}\times g:Y'\times T\ra Y'\times U$. 
Now we can compute
\[
\tilde \iota^\ast \mathscr V=\tilde \iota^\ast \varphi_{\overline T}^\ast\mathscr H\isom\varphi_T^\ast\overline\iota^\ast\mathscr H\isom
\varphi_T^\ast\varphi_{T\ast}\mathscr G\isom\mathscr G.
\]

This finishes the proof.
\end{proof}

\begin{cor}\label{alpino}
Let $\varphi:Y\ra Y'$ be an \'etale map of quasi-projective varieties, $C'\subset Y'$ a Cohen-Macaulay curve with preimage $C$.
Let $V\subset Q$ be the open subset parametrizing quotients $\mathscr I_C\surj F$ such that 
$\varphi(x)\neq \varphi(y)$ for all $x\neq y\in \Supp F$. Then there is an \'etale map $\Phi:V\ra Q'$.
\end{cor}

\begin{proof}
To apply Proposition \ref{prop:big}, we need the target to be complete. Therefore, after completing $Y'$ 
to a proper variety $\overline{Y'}$, let us denote by $\overline{C'}$ the scheme-theoretic closure of $C'$. 
Then, Proposition \ref{prop:big} gives us an \'etale map 
$\Phi:V\ra \overline{Q'}$, where the target is the scheme of length $n$ quotients of $\mathscr I_{\overline{C'}}$. 
The map sends $[F]\mapsto [\iota_\ast\varphi_\ast F]$, where $\iota:Y'\ra \overline{Y'}$ is the open immersion. 
However, the support of $\iota_\ast\varphi_\ast F$ can be identified with $\Supp (\varphi_\ast F)\subset Y'$ 
for all $[F]$, so that $\Phi$ actually factors through $Q'$.
\end{proof}

\subsection{Applications to threefolds}
In this section we assume $Y$ and $Y'$ are quasi-projective threefolds. 
All the other assumptions and notations from the previous sections remain unchanged here.

If $\varphi:Y\ra Y'$ is an \'etale map, we see that the induced morphism
\[
\Phi:V\ra Q'
\]
of Corollary \ref{alpino}, when restricted to the closed stratum $W^{(n)}_C\subset V$, appears in a Cartesian diagram
\be\label{crucial}
\begin{tikzcd}[row sep=large, column sep=large]
W^{(n)}_C\MySymb{dr}\arrow{r}{\pi_C} \arrow[swap]{d}{\Phi} & C\arrow{d}{\varphi}\\
W^{(n)}_{C'}\arrow{r}{\pi_{C'}} & C'
\end{tikzcd}
\ee
where the horizontal maps were defined in \eqref{map:zlt}.
Let $V'\subset Q'$ be the image of the \'etale map $\Phi:V\ra Q'$. Then the commutative diagram
\[
\begin{tikzcd}
W^{(n)}_C \arrow[hook]{r}\arrow[swap]{d}{\Phi} & V\arrow{d}{\textrm{\'et}}\arrow[hook]{r}{\textrm{open}} & Q\\
W^{(n)}_{C'}\arrow[hook]{r} & V' \arrow[hook]{r}{\textrm{open}} & Q'
\end{tikzcd}
\]
yields the relation 
\be\label{hgfcvb}
\nu_Q\big{|}_{W^{(n)}_C}=\Phi^\ast\bigl(\nu_{Q'}\big{|}_{W^{(n)}_{C'}}\bigr),
\ee
which will be useful in the next proof.

\begin{prop}\label{lekrjt}
Let $\varphi:Y\ra \mathbb A^3$ be an \'etale map of quasi-projective threefolds, and let $L\subset\mathbb A^3$ be a line. 
\bitem
\item [\emph{(i)}] If $C=\varphi^{-1}(L)\subset Y$, we have a natural isomorphism $W^{(n)}_C=C\times F_n$.
\item [\emph{(ii)}] The restricted Behrend function $\nu_Q\big{|}_{W^{(n)}_C}$ agrees with the pullback of $\nu_n$ 
under the natural projection to $F_n$.
\eitem
\end{prop}

\begin{proof}
With the help of \eqref{crucial}, we find a diagram
\[
\begin{tikzcd}[row sep=large, column sep=large]
& W^{(n)}_C\MySymb{dr}\arrow{r}{\pi_C}\arrow[swap]{d}{\Phi} & C\MySymb{dr}\arrow[hook]{r}{} \arrow[swap]{d}{} & Y\arrow{d}{\textrm{\'et}}\\
F_n & W^{(n)}_L\arrow{r}{\pi_L}\arrow[swap]{l}{p} & L\arrow[hook]{r} & \mathbb A^3
\end{tikzcd}
\]
so that the first claim follows by the isomorphism $W^{(n)}_L=L\times F_n$ of Proposition \ref{yuhjbn}. 
As for Behrend functions, we have, using \eqref{hgfcvb} and \eqref{nufunctions},
\[
\nu_Q\big{|}_{W^{(n)}_C}=\Phi^\ast\bigl(\nu_{M_n}\big{|}_{W^{(n)}_L}\bigr)=\Phi^\ast\bigl(p^\ast\nu_n\bigr).
\]
The claim follows.
\end{proof}

\medskip
The following can be viewed as the analogue of~\cite[Cor.~4.9]{BFHilb}.

\begin{cor}\label{pingu}
Let $Y$ be a smooth quasi-projective threefold. If $C\subset Y$ is a smooth curve, the map 
\[
\pi_C:W^{(n)}_C\ra C
\]
is a Zariski locally trivial fibration with fibre $F_n$.
More precisely, there exists a Zariski open covering $C_i\subset C$ such that for all $i$ one has an isomorphism
\be\label{strangedec}
(\pi_C^{-1}(C_i),\nu_Q)\isom (C_i,1)\times (F_n,\nu_n)
\ee
of schemes with constructible functions on them.
\end{cor}

\begin{proof}
Cover $Y$ with open affine subschemes $U_i$ such that, for each $i$, the closed immersion $C_i=C\cap U_i\subset U_i$ is given, when $C_i$ is nonempty, 
by the vanishing of two equations. We can do this because $C$ is a local complete intersection. 
Possibly after shrinking each $U_i$, we can find \'etale maps $U_i\ra \mathbb A^3$ and (using the smoothness of $C$) Cartesian diagrams
\[
\begin{tikzcd}[row sep=large, column sep=large]
C_i\MySymb{dr}\arrow[hook]{r}\arrow{d} & U_i\arrow{d}{\textrm{\'et}}\\
L\arrow[hook]{r} & \mathbb A^3
\end{tikzcd}
\]
where $L$ is a fixed line in $\mathbb A^3$. Combining \eqref{crucial} with (both statements of) Proposition \ref{lekrjt} yields Cartesian diagrams
\[
\begin{tikzcd}[row sep=large, column sep=large]
C_i\times F_n\MySymb{dr}\arrow{r}{\pi_{C_i}}\arrow[hook]{d} & C_i\arrow[hook]{d}\\
W^{(n)}_{C}\arrow{r}{\pi_C} & C
\end{tikzcd}
\]
and the claimed decomposition \eqref{strangedec}.
\end{proof}

\section{The weighted Euler characteristic of $Q^n_C$}\label{DTinvar}

The goal of this section is to prove the following result, anticipated in the Introduction.

\begin{teo}\label{thm:bfcomp}
Let $Y$ be a smooth quasi-projective threefold, $C\subset Y$ a smooth curve. If $Q^n_C=\Quot_n(\mathscr I_C)$, then
\[
\tilde\chi(Q^n_C)=(-1)^n\chi(Q^n_C).
\]
\end{teo}

\subsection{Ingredients in the proof}
We briefly discuss the main tools used in the proof of the above formula.

\subsubsection{Stratification}
We start by observing that we have a stratification
\begin{equation}\label{qwert}
Q^n_C=\coprod_{\substack{0\leq j\leq n \\ \alpha\vdash j}}\Hilb^{n-j}(Y\setminus C)\times W^\alpha_C
\end{equation}
by locally closed subschemes, ``separating'' the points \emph{away from} the curve from those embedded \emph{on} the curve.
We think of a partition $\alpha\vdash j$ as a tuple of positive integers
\[
\alpha_1\geq \cdots \geq \alpha_{r_{\alpha}}\geq 1
\]
such that $\sum\alpha_i=j$. Here $r_\alpha$ is the number of distinct parts of $\alpha$. Recall that 
\[
W^\alpha_C\subset Q^j_C,
\] 
defined for the first time in \eqref{def:strat}, 
parametrizes configurations of $r_{\alpha}$ distinct embedded points on $C$, 
having respective multiplicities $\alpha_1,\dots,\alpha_{r_{\alpha}}$.
According to \eqref{qwert}, it is natural to expect the number
\[
\tilde\chi(Q^n_C)=\chi(Q^n_C,\nu_{Q^n_C})
\]
to be computed combining the following data. 

First of all, contributions from $\Hilb^{n-j}(Y\setminus C)$ are taken care of by~\cite[Thm.~4.11]{BFHilb}, which implies the formula
\be\label{punctual}
\tilde\chi(\Hilb^k(Y\setminus C))=(-1)^k\chi(\Hilb^k(Y\setminus C)).
\ee
Secondly, contributions from $W^\alpha_C\subset W^j_C$ will be fully expressed (thanks to the content of the previous section) in terms of the 
deepest stratum. The only relevant character here is the ``punctual'' locus $F_n$. It will be enough to know that 
\be\label{localnu}
\chi(F_j,\nu_j)=(-1)^j\chi(F_j),
\ee
which follows from~\cite[Cor.~3.5]{BFHilb}. Note that here $\chi(F_j)=\chi(M_j)$ counts the number of fixed points
of the torus action we have recalled in \S~\ref{sectiontwo}.

\subsubsection{The Behrend function}\label{behfun}
According to~\cite{Beh}, any complex scheme $Z$ carries a canonical constructible function 
$\nu_Z:Z\ra\mathbb Z$ and one can consider the weighted Euler characteristic
\[
\tilde\chi(Z)=\chi(Z,\nu_Z)=\sum_{k\in\mathbb Z}k\chi(\nu_Z^{-1}(k)).
\]
Given a morphism $f:Z\ra X$, one also has the relative weighted Euler characteristic
\[
\tilde\chi(Z,X)=\chi(Z,f^\ast\nu_X).
\]
We now list its main properties following~\cite[Prop.~1.8]{Beh}.
First of all, it is clear that $\tilde\chi(Z)=\tilde\chi(Z,Z)$ through the identity map on $Z$.
\bitem
\item [(B1)]  If $Z=Z_1\amalg Z_2$ for $Z_i\subset Z$ locally closed, then 
\[
\tilde\chi(Z,X)=\tilde\chi(Z_1,X)+\tilde\chi(Z_2,X).
\]
\item [(B2)]  Given two morphisms $Z_i\ra X_i$, $i=1,2$, we have 
\[
\tilde\chi(Z_1\times Z_2,X_1\times X_2)=\tilde\chi(Z_1,X_1)\cdot \tilde\chi(Z_2,X_2).
\]  
\item [(B3)] Given a commutative diagram
\[
\begin{tikzcd} 
Z \arrow{r}{} \arrow{d}{}
&X \arrow{d}{}\\
W \arrow{r}{} &Y
\end{tikzcd}
\]
with $X\ra Y$ smooth and $Z\ra W$ finite \'etale of degree $d$, we have
\[
\tilde\chi(Z,X)=d(-1)^{\dim X/Y}\tilde\chi(W,Y).
\]
\item [(B4)] This is a special case of (B3): if $X\ra Y$ is \'etale (e.g.~an open immersion), then $\tilde \chi(Z,X)=\tilde\chi(Z,Y)$.
\eitem

\subsection{The computation}

We can start the proof of Theorem~\ref{thm:bfcomp}. Let us shorten $Y_0=Y\setminus C$ for convenience.
After fixing a partition $\alpha\vdash j$, let 
\[
V_\alpha\subset \prod_iQ^{\alpha_i}_C
\]
denote the open subscheme consisting of tuples $(F_1,\dots,F_{r_{\alpha}})$
of sheaves with pairwise disjoint support. According to Corollary \ref{alpino}, we can use the 
\'etale cover $\amalg_iY\ra Y$ to produce an \'etale morphism 
\[
f_\alpha:V_\alpha\ra Q^j_C.
\]
It is given on points by taking the ``union'' of the $0$-dimensional supports of the sheaves $F_i$.
Letting $U_\alpha$ be the image of $f_\alpha$, we can form the diagram
\[
\begin{tikzcd}[row sep=large]  
Z_\alpha\MySymb{dr}\arrow[hook]{r}{} \arrow[swap]{d}{\textrm{Galois}} & V_\alpha \arrow{d}{f_\alpha}\arrow[hook]{r}{\textrm{open}} & \prod_iQ^{\alpha_i}_C\\
W^\alpha_C \arrow[hook]{r}{} & U_\alpha\arrow[hook]{r}{\textrm{open}} & Q^j_C
\end{tikzcd}
\]
where the Cartesian square defines the scheme $Z_\alpha$. The morphism on the left is Galois with Galois group $G_\alpha$, the automorphism group of the partition $\alpha$.
It is easy to see that in fact
\[
Z_\alpha=\prod_iW^{(\alpha_i)}_C\setminus \Delta
\]
also fits in the Cartesian square
\begin{equation}\label{mistery}
\begin{tikzcd}[row sep=large] 
Z_\alpha\MySymb{dr}\arrow[hook]{r}{\textrm{open}} \arrow{d}{}
&\prod_iW^{(\alpha_i)}_C \arrow{d}{\pi_\alpha}\\
C^{r_{\alpha}}\setminus \Delta \arrow[hook]{r}{\textrm{open}} &C^{r_{\alpha}}
\end{tikzcd}
\end{equation}
where $W^{(\alpha_i)}_C\subset Q^{\alpha_i}_C$ is the deep stratum, $\Delta$ denotes 
the ``big diagonal'' (where at least two entries are equal), 
and the vertical map $\pi_\alpha$ is the product of the fibrations 
$\pi_C:W^{(\alpha_i)}_C\ra C$, for $i=1,\dots,r_\alpha$.

\medskip
We need two identities before we can finish the computation.

\medskip
\noindent
\textbf{First identity.} We have
\be\label{lem:chiwalpha}
\chi(W^\alpha_C)=|G_\alpha|^{-1}\chi(C^{r_{\alpha}}\setminus \Delta)\prod_i\chi(F_{\alpha_i}).
\ee

Indeed, for each $\alpha$, the map
\[
\pi_\alpha:Z_\alpha \ra C^{r_{\alpha}}\setminus \Delta
\]
appearing in \eqref{mistery} is Zariski locally trivial with fiber $\prod_iF_{\alpha_i}$ by Corollary \ref{pingu}.
Formula \eqref{lem:chiwalpha} follows since $W^\alpha_C$ is the free quotient $Z_\alpha/G_\alpha$.

\medskip
\noindent
\textbf{Second identity.} We have 
\be\label{eq:gna}
\tilde\chi\bigl(Z_\alpha,\prod_iQ^{\alpha_i}_C\bigr)=\chi(C^{r_{\alpha}}\setminus \Delta) \prod_i\chi(F_{\alpha_i},\nu_{\alpha_i}).
\ee

Indeed, by Corollary \ref{pingu}, we can find a Zariski open cover $\{B_s\}_s$ of $C^{r_{\alpha}}\setminus \Delta$
such that 
\[
(\pi_\alpha^{-1}B_s,\nu)\isom (B_s,1_{B_s})\times \Bigl(\prod_iF_{\alpha_i},\prod_i\nu_{\alpha_i}\Bigr).
\]
In the left hand side, $\nu$ denotes the Behrend function restricted from $\prod_iQ^{\alpha_i}_C$. 
We can refine this to a locally closed stratification 
$\amalg_\ell U_\ell=C^{r_{\alpha}}\setminus \Delta$ such that each $U_\ell$ is contained in some $B_s$. Therefore,
\begin{align*}
\tilde\chi\bigl(Z_\alpha,\prod_iQ^{\alpha_i}_C\bigr)&=\sum_\ell \tilde\chi\bigl(\pi_\alpha^{-1}U_\ell,\prod_iQ^{\alpha_i}_C\bigr) 
\qquad\qquad\,\,\,\,\,\, \textrm{\small{by (B1)}}\\
&=\sum_\ell\chi\bigl(U_\ell\times \prod_iF_{\alpha_i},1_{U_\ell}\times\prod_i \nu_{\alpha_i}\bigr)\\
&=\sum_\ell\chi(U_\ell,1_{U_\ell}) \prod_i\chi(F_{\alpha_i},\nu_{\alpha_i})
\qquad\,\,\,\, \textrm{\small{by (B2)}}\\
&=\chi(C^{r_{\alpha}}\setminus \Delta) \prod_i\chi(F_{\alpha_i},\nu_{\alpha_i}),
\end{align*}
and \eqref{eq:gna} is proved.

\medskip
Note that combining \eqref{qwert} and \eqref{lem:chiwalpha} we get
\be\label{euler}
\chi(Q^n_C)=\sum_{j,\alpha}\chi(\Hilb^{n-j}Y_0)\cdot |G_\alpha|^{-1}\chi(C^{r_{\alpha}}\setminus \Delta)\prod_i\chi(F_{\alpha_i}).
\ee

We now have all the tools to finish the computation.
Let us fix $j$ and a partition $\alpha\vdash j$. We define
\[
D_\alpha\subset \Hilb^{n-j}Y\times \prod_iQ^{\alpha_i}_C
\]
to be the set of tuples $(Z_0,F_1,\dots,F_{r_{\alpha}})$ such that $(F_1,\dots,F_{r_{\alpha}})\in V_\alpha$ and the support of
$Z_0$ does not meet the support of any $F_i$. Then $D_\alpha$ is an open subscheme. 
The Galois cover $1\times f_\alpha:\Hilb^{n-j}Y_0\times Z_\alpha\ra \Hilb^{n-j}Y_0\times W^\alpha_C$ extends to an \'etale map
$D_\alpha\ra Q^n_C$, so that we have a commutative diagram
\be\label{comm1}
\begin{tikzcd}[row sep=large] 
\Hilb^{n-j}Y_0\times Z_\alpha \arrow[hook]{r}{} \arrow[swap]{d}{1\times f_\alpha}
&D_\alpha \arrow{d}{\textrm{\'et}}\\
\Hilb^{n-j}Y_0\times W^\alpha_C \arrow[hook]{r}{} & Q^n_C.
\end{tikzcd}
\ee
Therefore we can start computing $\tilde\chi(Q^n_C)=\chi(Q^n_C,\nu_{Q^n_C})$ as follows:
\begin{align*}
\tilde\chi(Q^n_C)&=\sum_{j,\alpha}\tilde\chi(\Hilb^{n-j}Y_0\times W^\alpha_C,Q^n_C)
\qquad\qquad\,\,\small{\textrm{by (B1) applied to \eqref{qwert}}} \\
&=\sum_{j,\alpha}|G_\alpha|^{-1}\tilde\chi(\Hilb^{n-j}Y_0\times Z_\alpha,D_\alpha)\qquad\small{\textrm{by (B3) applied to \eqref{comm1}}}\\
&=\sum_{j,\alpha}|G_\alpha|^{-1}\tilde\chi\bigl(\Hilb^{n-j}Y_0\times Z_\alpha,\Hilb^{n-j}Y\times \prod_iQ^{\alpha_i}_C\bigr) 
\,\,\,\,\,\,\,\,\,\,\,\small{\textrm{by (B4)}}\\
&=\sum_{j,\alpha}|G_\alpha|^{-1}\tilde\chi\bigl(\Hilb^{n-j}Y_0,\Hilb^{n-j}Y\bigr)\cdot 
\tilde\chi\bigl(Z_\alpha,\prod_iQ^{\alpha_i}_C\bigr)\,\,\,\,\,\,\,\,\,\,\small{\textrm{by (B2)}}\\
&=\sum_{j,\alpha}|G_\alpha|^{-1}\tilde\chi(\Hilb^{n-j}Y_0)\cdot \chi(C^{r_{\alpha}}\setminus \Delta)
\prod_i\chi(F_{\alpha_i},\nu_{\alpha_i}) \qquad \small{\textrm{by (B4) and~\eqref{eq:gna}}}\\
&=(-1)^n\sum_{j,\alpha}\chi(\Hilb^{n-j}Y_0)\cdot |G_\alpha|^{-1}\chi(C^{r_{\alpha}}\setminus\Delta) 
\prod_i\chi(F_{\alpha_i}) \,\,\,\,\, \small{\textrm{by \eqref{punctual} and \eqref{localnu}}}\\
&=(-1)^n\chi(Q^n_C) \qquad \small{\textrm{by~\eqref{euler}}}.
\end{align*}
This completes the proof of Theorem~\ref{thm:bfcomp}.

\begin{question} 
It would be nice to determine whether the Behrend function on $M_n=\Quot_n(\mathscr I_L)$ 
is the constant sign $(-1)^n$. As far as we know, this is still open even when the curve is absent, i.e.~for $\Hilb^n\mathbb A^3$. 
\end{question}

\section{Ideals, pairs and quotients}\label{dtptwallcrossing}

In this section we give some applications of the formula
\[
\tilde\chi(Q^n_C)=(-1)^n\chi(Q^n_C).
\]
We show that the DT/PT correspondence holds for the contribution 
of a smooth \emph{rigid} curve in a projective Calabi-Yau threefold. 
We discuss, at a conjectural level, the case of an arbitrary smooth curve.

\subsection{Local contributions}
We fix a smooth projective threefold $Y$ and a 
Cohen-Macaulay curve $C\subset Y$ 
of arithmetic genus $g=1-\chi(\mathscr O_C)$, embedded in class $\beta\in H_2(Y,\mathbb Z)$. 
We will use the Quot scheme to endow the closed subset
\[
\set{Z\subset Y|C\subset Z,\,\chi(\mathscr I_C/\mathscr I_Z)=n}\subset I_{1-g+n}(Y,\beta)
\]
with a natural scheme structure.

\begin{lemma}\label{closedIMM}
There is a closed immersion $\iota:Q^n_C\ra I_{1-g+n}(Y,\beta)$.
\end{lemma}

\begin{proof}
Let $\mathscr I_{C\times T}\surj\mathscr F$ be a flat family of quotients parametrized by a scheme $T$.
Letting $Z\subset Y\times T$ be the subscheme defined by the kernel of the surjection, we get an exact sequence
\[
0\ra \mathscr F\ra \mathscr O_Z\ra \mathscr O_{C\times T}\ra 0.
\]
The middle term is flat over $T$, therefore it determines a point in the Hilbert scheme of $Y$. The discrete invariants $\beta$ and $\chi=1-g+n$
are the right ones, as can be seen by restricting the above short exact sequence to closed points of $T$. Therefore we get a morphism
\[
\iota:Q^n_C\ra I_{1-g+n}(Y,\beta).
\]
The correspondence at the level of functor of points is injective, and the morphism is proper (since the Quot scheme is proper,
as $Y$ is projective). Therefore $\iota$ is a closed immersion.
\end{proof}

\begin{defin}\label{defquot}
We define
\be\label{localmoduli}
I_n(Y,C)\subset I_{1-g+n}(Y,\beta)
\ee
to be the scheme-theoretic image of $\iota:Q^n_C\ra I_{1-g+n}(Y,\beta)$.
\end{defin}

\begin{rem}
The closed subset $|I_n(Y,C)|\subset I_{1-g+n}(Y,\beta)$ also has a scheme structure induced by 
GIT wall-crossing~\cite{ST}. Another scheme structure is defined in the recent paper~\cite{BrKo}. See in particular
Definition 4, where the notation used is $\Hilb^n(Y,C)$. We believe both these scheme structures agree with 
the one of our Definition \ref{defquot}, in which case they describe schemes isomorphic to $Q^n_C$. 
\end{rem}

Assume $Y$ is a projective \emph{Calabi-Yau} threefold.
By the main result of~\cite{Beh}, the degree $\beta$ curve counting invariants
\[
\DT_{m,\beta}=\int_{[I_m(Y,\beta)]^\vir}1,\qquad \PT_{m,\beta}=\int_{[P_m(Y,\beta)]^\vir}1
\]
can be computed as weighted Euler characteristics of the corresponding moduli spaces, since the obstruction theories
defining the virtual cycles are symmetric. One can define the contribution of $C$ to the above invariants as
\be\label{contrib}
\DT_{n,C}=\chi(I_n(Y,C),\nu_I),\qquad \PT_{n,C}=\chi(P_n(Y,C),\nu_P).
\ee
Here we have set $I=I_{1-g+n}(Y,\beta)$ and $P=P_{1-g+n}(Y,\beta)$. The subscheme $P_n(Y,C)\subset P$ consists of stable pairs with Cohen-Macaulay support equal to $C$.
Note that these integers remember how $C$ sits inside $Y$, since the weight is the Behrend function coming from the full moduli space.

An immediate consequence of Theorem~\ref{thm:bfcomp} is a formula for the DT contribution of a smooth rigid curve.

\begin{teo}\label{rigidthm}
Let $Y$ be a projective Calabi-Yau threefold, $C\subset Y$ a smooth rigid curve. Then
\[
\DT_{n,C}=(-1)^n\chi(I_n(Y,C)).
\]
\end{teo}

\begin{proof}
The inclusion \eqref{localmoduli} is both open and closed thanks to the infinitesimal isolation of $C$. Then $\nu_I|_{I_n(Y,C)}=\nu_{I_n(Y,C)}$, thus
\[
\DT_{n,C}=\tilde\chi(I_n(Y,C))=(-1)^n\chi(I_n(Y,C)),
\]
as claimed. 
\end{proof}

\begin{rem}
In the rigid case, $\DT_{n,C}$ is a DT invariant in the classical sense,
namely it is the degree of the virtual class $[I_n(Y,C)]^\vir$ 
obtained by restricting the one on $I_{1-g+n}(Y,\beta)$.
\end{rem}

\medskip
Theorem \ref{rigidthm} can be seen as an instance of the following more general result,
which is also a direct consequence of Theorem~\ref{thm:bfcomp}.

\begin{prop}\label{corgenseries}
Let $Y$ be a smooth projective threefold. If $C\subset Y$ is a smooth curve of genus $g$, then
\be\label{dfghj}
\sum_{n\geq 0}\tilde\chi(I_n(Y,C))q^n=M(-q)^{\chi(Y)}(1+q)^{2g-2}.
\ee
\end{prop}

\begin{proof}
For any smooth threefold $X$ we have Cheah's formula~\cite{Cheah}
\[
\sum_{n\geq 0}\chi(\Hilb^n X)q^n=M(q)^{\chi(X)}.
\]
We use this with $X=Y_0=Y\setminus C$, together with formula~\eqref{euler}, to compute
\begin{align*}
\sum_{n\geq 0}\chi(I_n(Y,C))q^n&=M(q)^{\chi(Y\setminus C)}\cdot \Bigl(\sum_{n\geq 0}\chi(F_n)q^n\Bigr)^{\chi(C)} \\
&=M(q)^{\chi(Y\setminus C)}\cdot \Bigl(\sum_{n\geq 0}\chi(M_n)q^n\Bigr)^{\chi(C)} \\
&=M(q)^{\chi(Y\setminus C)}\cdot \Bigl(\frac{M(q)}{1-q}\Bigr)^{\chi(C)}\qquad\textrm{by \eqref{localeuler}}\\
&=M(q)^{\chi(Y)}(1-q)^{2g-2}.
\end{align*}
The claimed formula follows by Theorem \ref{thm:bfcomp}.
\end{proof}

\begin{rem}
Formula \eqref{dfghj} can be rewritten as
\be\label{alskdjqpwoei}
\sum_{n\geq 0}\tilde\chi(I_n(Y,C))q^n=M(-q)^{\chi(Y)}\sum_{n\geq 0}\tilde\chi(P_n(Y,C))q^n.
\ee
Indeed $P_n(Y,C)=\Sym^nC$ is smooth of dimension $n$, thus $\tilde\chi=(-1)^{n}\chi$. 
The latter identity can be seen as the $\nu$-weighted version of the ``local'' wall-crossing formula
between ideals and stable pairs, which was already established for a single Cohen-Macaulay
curve at the level of Euler characteristics~\cite[Thm.~1.5]{ST}. In other words, \eqref{alskdjqpwoei}
is precisely what happens to the Stoppa-Thomas identity
\[
\sum_{n\geq 0}\chi(I_n(Y,C))q^{n}=M(q)^{\chi(Y)}\sum_{n\geq 0}\chi(P_n(Y,C))q^{n}
\]
when we replace $q$ by $-q$.
\end{rem}

\subsection{DT/PT wall-crossing at a single curve}
Let $C$ be a smooth curve of genus $g$, embedded 
in class $\beta$ in a smooth projective Calabi-Yau threefold $Y$. Let us define the generating series
\begin{align*}
\mathsf{DT}_C(q)&=\sum_{n\geq 0}\DT_{n,C}q^{n}\\
\mathsf{PT}_C(q)&=\sum_{n\geq 0}\PT_{n,C}q^{n}
\end{align*}
encoding the local contributions defined in~\eqref{contrib}. The stable pair side 
has already been computed~\cite[Lemma~3.4]{BPS}. The result is 
\be\label{bpscorrection}
\mathsf{PT}_C(q)=n_{g,C}\cdot (1+q)^{2g-2},
\ee
where $n_{g,C}$ is the $g$-th BPS number of $C$. For instance,
if $C$ is rigid, then $n_{g,C}=1$ and thanks to Theorem \ref{rigidthm} we see that \eqref{dfghj} can be rewritten as
\[
\mathsf{DT}_C(q)=M(-q)^{\chi(Y)}\cdot \mathsf{PT}_C(q).
\]
This formula can be seen as a ``local DT/PT correspondence'', or local wall-crossing formula at $C$.
We next prove that such formula, for arbitrary $C$, is equivalent to the following conjecture.

\begin{conj}\label{conj1}
Let $C$ be a smooth curve in a projective Calabi-Yau threefold $Y$. 
Let $\mathcal I=I_{1-g}(Y,\beta)$ be the Hilbert scheme where the ideal sheaf of $C$ lives as a point. Then, for all $n$, one has
\[
\DT_{n,C}=\nu_{\mathcal I}(\mathscr I_C)\cdot \tilde\chi(I_n(Y,C)).
\]
\end{conj}

\begin{rem}
An equivalent formula has been conjectured by Bryan and Kool in their recent paper~\cite{BrKo}.
See Conjecture 18 in \emph{loc.~cit.}~for the precise (more general) setting.
\end{rem}

\begin{teo}\label{thm:wc}
Let $Y$ be a projective Calabi-Yau threefold, $C\subset Y$ a smooth curve. Then Conjecture \ref{conj1} is equivalent to the wall-crossing identity
\[
\mathsf{DT}_C(q)=M(-q)^{\chi(Y)}\cdot \mathsf{PT}_C(q).
\]
\end{teo}

\begin{proof}
Combining~\eqref{bpscorrection} with~\eqref{dfghj}, we see that the right hand side of the formula equals
\[
n_{g,C}\cdot \sum_{n\geq 0}\tilde\chi(I_n(Y,C))q^{n}.
\] 
Therefore the DT/PT correspondence holds at $C$ if and only if
\[
\DT_{n,C}=n_{g,C}\cdot\tilde\chi(I_n(Y,C)).
\]
We are then left with proving that $\nu_{\mathcal I}(\mathscr I_C)=n_{g,C}$. 
Recall that the moduli space of ideal sheaves is isomorphic to the moduli space
of stable pairs along the \emph{open} subschemes parametrizing \emph{pure} curves.
Moreover, the map $\phi:P_{1-g}(Y,\beta)\ra \mathcal M$ to the moduli space of stable pure 
sheaves considered in~\cite{BPS}, defined by forgetting the section of a stable pair, satisfies the relation 
\[
\nu_{P_{1-g}(Y,\beta)}=(-1)^g\phi^\ast\nu_{\mathcal M}
\]
by~\cite[Thm.~4]{BPS}. Hence
\begin{align*}
\nu_{\mathcal I}(\mathscr I_C)&=\nu_{\mathcal I^\pur}(\mathscr I_C)\\
&=\nu_{P_{1-g}(Y,\beta)}([\mathscr O_Y\surj \mathscr O_C])\\
&=(-1)^g\nu_{\mathcal M}(\mathscr O_C)\\
&=n_{g,C}
\end{align*}
where the last equality is~\cite[Prop.~3.6]{BPS}.
\end{proof}

\begin{rem}
Thanks to the identity $\nu_{\mathcal I}(\mathscr I_C)=n_{g,C}$, proved in the course of Theorem~\ref{thm:wc}, 
Conjecture~\ref{conj1} can be rephrased as
\[
\DT_{n,C}=\nu_{P}|_{P_n(Y,C)} \cdot \chi(I_n(Y,C)),
\] 
where $\nu_{P}|_{P_n(Y,C)}$ is the constant $(-1)^n\cdot n_{g,C}=(-1)^{n-g}\nu_{\mathcal M}(\mathscr O_C)$. In particular the conjecture says that the DT and PT contributions of $C$
differ from the Euler characteristic of the corresponding moduli space by the \emph{same} constant.
\end{rem}

We end the paper with some speculations, indicating plausibility reasons 
why Conjecture \ref{conj1} should hold true.

Suppose we were able to show that, given a point $\mathscr I_Z\in I_n(Y,C)\subset I$, a formal neighborhood 
of $\mathscr I_Z$ in $I$ is isomorphic to a product
\[
U\times V,
\]
where $U$ is a formal neighborhood of $\mathscr I_C$ in $\mathcal I$ and $V$ is a formal neighboorhood of $\mathscr I_Z$ in $I_n(Y,C)$.
Then, since the Behrend function value $\nu(P)$ only depends on a formal neighborhood of $P$~\cite{Jiang}, this would immediately lead to the Behrend
function identity
\be\label{stronger}
\nu_I|_{I_n(Y,C)}=\nu_{\mathcal I}(\mathscr I_C)\cdot \nu_{I_n(Y,C)},
\ee
from which Conjecture~\ref{conj1} follows after integration.
One reason to believe in a product decomposition as above is the following. At least when the maximal purely $1$-dimensional part $C\subset Z$
is smooth, one may expect to be able to ``separate'' infinitesimal deformations of $C$ (the factor $U$) from those deformations of $Z$ that keep $C$ fixed 
(the factor $V$ in the Quot scheme). This decomposition is manifestly false when $C$ 
acquires a singularity, and we do not know of any counterexample in the smooth case.

\bigskip

{\noindent{\bf Acknowledgements.}
First, I wish to thank my advisors Martin G.~Gulbrandsen and Lars H.~Halle for their help and the many insightful conversations around the topics
discussed here. I also thank Jim Bryan, Letterio Gatto, Martijn Kool, Jacopo Stoppa and Richard Thomas for valuable comments and for sharing their ideas on the subject.
}

\clearpage
\bibliographystyle{amsalpha}
\bibliography{bib}

\providecommand{\bysame}{\leavevmode\hbox to3em{\hrulefill}\thinspace}
\providecommand{\MR}{\relax\ifhmode\unskip\space\fi MR }
\providecommand{\MRhref}[2]{%
  \href{http://www.ams.org/mathscinet-getitem?mr=#1}{#2}
}
\providecommand{\href}[2]{#2}
\begin{thebibliography}{MNOP06}

\bibitem[BB07]{BB}
K.~Behrend and J.~Bryan, \emph{{\itshape{Super-rigid Donaldson-Thomas
  invariants}}}, Math. Res. Lett. \textbf{14} (2007), 559--571.

\bibitem[Beh09]{Beh}
K.~Behrend, \emph{{Donaldson-Thomas type invariants via microlocal geometry}},
  Ann. of Math. \textbf{2} (2009), no.~170, 1307--1338.

\bibitem[BF08]{BFHilb}
K.~Behrend and B.~Fantechi, \emph{{\itshape{Symmetric obstruction theories and
  Hilbert schemes of points on threefolds}}}, Algebra Number Theory \textbf{2}
  (2008), 313--345.

\bibitem[BK16]{BrKo}
J.~Bryan and M.~Kool, \emph{Donaldson-thomas invariants of local elliptic
  surfaces via the topological vertex}, preprint (2016).

\bibitem[Bri11]{Bri}
T.~Bridgeland, \emph{Hall algebras and curve counting invariants}, J. Amer.
  Math. Soc. \textbf{24} (2011), no.~4, 969--998.

\bibitem[Che96]{Cheah}
J.~Cheah, \emph{On the cohomology of {H}ilbert schemes of points}, J. Algebraic
  Geom. \textbf{5} (1996), no.~3, 479--511.

\bibitem[GP99]{GP}
T.~Graber and R.~Pandharipande, \emph{Localization of virtual classes}, Invent.
  Math. \textbf{135} (1999), no.~2, 487--518.

\bibitem[Jia]{Jiang}
Y.~Jiang, \emph{Motivic {M}ilnor fiber of cyclic ${L}_\infty$-algebras}, to
  appear in Acta Math. Sin.

\bibitem[Li06]{JLI}
J.~Li, \emph{Zero dimensional {D}onaldson-{T}homas invariants of threefolds},
  Geom. Topol. \textbf{10} (2006), 2117--2171.

\bibitem[LP09]{LEPA}
M.~Levine and R.~Pandharipande, \emph{Algebraic cobordism revisited}, Invent.
  Math. \textbf{176} (2009), no.~1, 63--130.

\bibitem[MNOP06]{MNOPpaper}
D.~Maulik, N.~Nekrasov, A.~Okounkov, and R.~Pandharipande, \emph{{Gromov-Witten
  theory and Donaldson-Thomas theory I}}, Compos. Math. \textbf{142} (2006),
  1263--1285.

\bibitem[Nit05]{Nit}
N.~Nitsure, \emph{Construction of {H}ilbert and {Q}uot schemes}, Fundamental
  Algebraic Geometry, Math. Surveys Monogr., vol. 123, Amer. Math. Soc.,
  Providence, RI, 2005, pp.~105--137.

\bibitem[PT09]{PT}
R.~Pandharipande and R.~P. Thomas, \emph{Curve counting via stable pairs in the
  derived category}, Invent. Math. \textbf{178} (2009), no.~2, 407--447.

\bibitem[PT10]{BPS}
\bysame, \emph{Stable pairs and {BPS} invariants}, J. Amer. Math. Soc.
  \textbf{23} (2010), no.~1, 267--297.

\bibitem[ST11]{ST}
J.~Stoppa and R.~P. Thomas, \emph{Hilbert schemes and stable pairs:
  \uppercase{git} and derived category wall crossings}, Bulletin de la
  Soci\'et\'e Math\'ematique de France \textbf{139} (2011), no.~3, 297--339.

\bibitem[Tod10]{Toda1}
Y.~Toda, \emph{Curve counting theories via stable objects {I}. {DT}/{PT}
  correspondence}, J. Amer. Math. Soc. \textbf{23} (2010), no.~4, 1119--1157.

\end{thebibliography}

\end{document}